\documentclass[11pt]{amsart}
\usepackage{amssymb}
\usepackage{thmtools}
\usepackage{thmtools}
\usepackage{enumitem}
\usepackage[
  pdftex,
  colorlinks=true,
  urlcolor=blue,       
  filecolor=blue,     
  linkcolor=blue,       
  citecolor=blue,         
  pdftitle={A polynomial formula for the perspective four points problem},
  pdfauthor={David Lehavi, Brian Osserman},
  pdfsubject={},
  pdfkeywords={}
  pagebackref,
  pdfpagemode=None,
  bookmarksopen=true]{hyperref}
\usepackage{hyperref}
\usepackage{color}
\usepackage{cleveref}[capitalize]
\newcommand{\BR}{{\mathbb{R}}}

\newcommand{\BP}{{\mathbb{P}}}
\newcommand{\cU}{{\mathcal{U}}}

\DeclareMathOperator{\orig}{orig}
\theoremstyle{plain}
\newtheorem{lma}{Lemma}[section]
\newtheorem{thm}[lma]{Theorem}

\newtheorem{prp}[lma]{Proposition}
\newtheorem{cor}[lma]{Corollary}
\theoremstyle{definition}

\newtheorem{dfn}[lma]{Definition}
\newtheorem{ex}[lma]{Example}
\newtheorem{rmr}[lma]{Remark}

\setlist[description]{font=\normalfont\itshape-\space}
\begin{document}
\title[A polynomial formula for P4P]{A polynomial formula for the perspective four points problem}
\begin{abstract}
We present a fast and accurate solution to the perspective $n$-points problem, by way of a new approach to the $n=4$ case. Our solution
hinges on a novel separation of variables: given four 3D points and four corresponding 2D points
on the camera canvas, we start by finding another
set of 3D points, sitting on the rays connecting the camera to the 2D
canvas points, so that the six pair-wise distances between these 3D points
are as close as possible to the six distances between the original 3D points. This step reduces
the perspective problem to an absolute orientation problem, which has
a solution via explicit formula. To solve the first problem we set coordinates
which are as orientation-free as possible: on the 3D points side our
coordinates are the squared distances between the points. On the 2D
canvas-points side our coordinates are the dot products of
the points after rotating one of them to sit on the optical axis. We then
derive the solution with the help of a computer algebra system. Our solution is an order of
magnitude faster than state of the art algorithms, while offering similar accuracy under
realistic noise. Moreover, our reduction to the absolute orientation problem
runs two orders of magnitude faster than other perspective problem solvers,
allowing extremely efficient seed rejection when implementing RANSAC.
\end{abstract}
\author{David Lehavi}
\email{dlehavi@gmail.com}
\author{Brian Osserman}
\email{brian.osserman@gmail.com }
\date{\today}
\maketitle
%
\section{Introduction}
%
The perspective $n$ point problem (PnP) goes back to the mid 19th-century, and seeks to recover the
6DoF pose of a calibrated camera given $n$ points in the world, and their images on a camera canvas.
The problem
arises as a localization problem in most computer visions problems involving both images and 3D data.
Typically in these problems one starts with a large number of pairings between 2D and 3D points, 
where unfortunately each pairing is correct with relatively low probability; one then uses many subsets
of size $3$ or $4$ of these pairings as the ``seeds'' or ``RANdom SAmple'' part of a RANSAC on
the entire set of pairings.

In this paper we propose a novel solution to the $n=4$ problem, which is an order of magnitude
faster than previously suggested solutions, while providing accuracy close to state of the art.
Moreover, our algorithm provides an accuracy measure which enables us to 
efficiently reject low-accuracy seeds and then unite compatible seeds before even solving for the
actual pose; this part runs fully two orders of magnitude faster than
existing solutions. A portion of this efficiency is due to the fact that
our algorithm consists almost entirely of evaluating algebraic formulas,
so it is nearly devoid of branches, which makes it extremely
amenable to SIMD implementation.

Our key contribution is an extremely fast algorithm which, given four marked
3D points, and four 2D points on the camera plane with corresponding markings,
reduces the perspective problem to an absolute orientation problem as follows:
one may first rotate
the 2D canvas so that the last point would be on the camera axis, then
compute the six distances between the 3D points as well as the six dot
products between the 2D points, reprojected to the rotated canvas. Next one
feeds these values into {\em explicit algebraic formulas}, yielding four
estimated depths associated to the 2D points in the rotated system. Finally,
the depths are rescaled to obtain depths relative to the original,
non-rotated camera plane. This is our intermediate output, which is in effect
another configuration of 3D points, and thus reduces the original perspective
problem to an absolute orientation problem. This reduction step runs
two orders of magnitude faster than standard algorithms.

Having achieved the intermediate output,
we can reject seeds having high error, and then unite low-error seeds
when they yield the same set of
$z$-depths for a common triple of points. This is all done prior to solving
the absolute orientation problem (using e.g. Horn's algorithm -- see
\cite{Ho}), which yields additional efficiency compared to approaches that
require solving for the pose on each seed separately. Finally, we use
Fletcher's variant of the Levenberg-Marquardt gradient descent algorithm
\cite{F} to fine-tune the solution and minimize the reprojection error.
Another benefit of uniting seeds before solving for pose is that simply by
prioritizing seeds which can be united with large numbers of other seeds,
our approach will often implicitly discard seeds which have very low
apparent error but are in fact based on mismatched points, since these
typically can't be united with as many other seeds.

The crux of why our reduction is feasible lies in the fact that we are
reducing the number of parameters involved in both the problem input, and the
problem output. Instead of directly representing the coordinates of each of
the four input 3D points and the four 2D canvas points, which would require
$(3+2)\cdot 4 = 20$ numbers, our input consists of the $12$ squared distances
and dot products described above. Likewise, our intermediate output consists
of four depths, instead of the $4+3$ numbers typically
needed to represent rotation via quaternionic representation and translation.
This concise representation of data enables us to obtain the aforementioned
explicit algebraic formula for the output in terms of the input. We obtained
this formula with the assistance of the computer algebra system
\verb|Singular|, which is in itself
remarkable, since algebra problems arising from real-world problems and
involving so many variables are usually well beyond the reach of computer
algebra packages. Even with our reduced variable count, \verb|Singular|
was unable to solve the problem in the most direct way, and we had to use
some {\em ad hoc} experimentation with different approaches in order to
obtain our ultimate formulas. We emphasize however that the computer algebra
system was used only in deriving our algorithm; our algorithm itself requires
nothing more sophisticated than evaluating multivariate polynomials and
taking square roots.

%
\section{Related work}\label{s:related}
%
The $n=3$ and $n=4$ cases are conceptually different: the case $n=3$ (which goes back to
\cite{G}) admits four pairs of solutions for
nondegenerate configurations. Since \cite{G} various ways have been suggested to solve this
case, all essentially boiling down to how to find a bi-quartic polynomial representing the
solutions.

In the case $n>3$ the perspective $n$ points problem is overdetermined, which means
that putatative solution can come with some accuracy measures which may help
fast rejection - which is useful for algorithms which come before a RANSAC algorithm.
Indeed, the prevalent algorithms for $n>3$ all come -- at least internally -- with such
a measure.
Broadly speaking, one may classify perspective $n>3$ points solvers along two dimensions:
\begin{enumerate}
\item which quantity is one trying to minimize or maximize? E.g., reprojection error, reconstruction error, error in the space of the rotations, or likelihood of the position being correct;
\item which optimization strategy is employed? E.g., finding a kernel of some matrix, least squares, or direct solution via formula.
\end{enumerate}

The most influential work in the problem domain
is ``EPnP'' \cite{LMFP}, which minimizes distances in the space of
quaternions after algebraically eliminating the translation; the optimization strategy is
finding the null space of a certain matrix.
The ``gold standard'' for accuracy with reasonable run time is ``SQPnP'' \cite{TL} which
minimizes the reconstruction error, which is the difference between the given 3D configuration, and another
putative 3D configuration lying on the rays along the 2D configuration. As far as the authors are
aware, these are the two most common algorithms for tackling the problem for $n>3$; e.g. they
are the ones supported in OpenCV.

Some other standard methods are \cite{LHM} which
optimizes error in the space of linear transformation by iterating a least squares solution,
\cite{HR} which introduces a ``relaxed''
reprojection error equation, for which the maximum likelihood of the position can be solved, and
\cite{ZKSAO} which builds a formula for minimizing
errors in the space of rotations.

In contrast, instead of trying to minimize the reconstruction error our work
achieves a formula which solves the reconstruction problem when there is
no noise. This turns out to be extremely fast, with accuracy that is broadly
comparable to EPnP and SQPnP; see \S \ref{s:experiments} below for a
detailed comparison.

%
\section{Our four-point algorithm}\label{s:algorithm}
%

We begin by describing our algorithm for solving the four-point perspective
problem. Accordingly, we suppose we have been given a quadruple 
$$(P^{\orig}_i, p^{\orig}_i)_{i=0,\dots,3}$$
of pairs where each $P^{\orig}_i$ is a point in 3D space, and $p^{\orig}_i$ is a canvas point
which has been matched to $P^{\orig}_i$. We consider the canvas as being imbedded
in space on the $z=1$ plane, and our only constraint is that 
$p^{\orig}_i \cdot p^{\orig}_3 \neq 0$ for each $i$; this arises from the normalization we
use in the derivation of our formulas.

Our goal is to find $z$-depths for $p^{\orig}_i$,
which we will denote by $z_i$, such that the distances between the $P^{\orig}_i$
agree with the distances between the $z_i p^{\orig}_i$. Our algorithm accomplishes
this in the following steps:

\begin{enumerate}
\item compute new coordinates for the tuple $(P^{\orig}_i,p^{\orig}_i)_i$ in terms of
distances and dot products as follows:
\begin{align*} a_i & = ||P^{\orig}_j - P^{\orig}_k||^2, \quad
	& b_i = \frac{(p^{\orig}_i \cdot p^{\orig}_i)(p^{\orig}_3 \cdot p^{\orig}_3)}{(p^{\orig}_i \cdot p^{\orig}_3)^2}, \\
c_i & = ||P^{\orig}_i - P^{\orig}_3||^2, \quad
	& d_i = \frac{(p^{\orig}_j \cdot p^{\orig}_k)(p^{\orig}_3 \cdot p^{\orig}_3)}{(p^{\orig}_j \cdot p^{\orig}_3)(p^{\orig}_k \cdot p^{\orig}_3)},
\end{align*}
where $i,j,k$ are all between $0$ and $2$ and the $j,k$ are determined by 
{\em cyclic indexing}:
$$j\equiv i + 1\mod3,\qquad k\equiv j + 1\mod 3;$$
\item use the $a_i, b_i, c_i, d_i$ and the formulas in
Appendix \ref{app:coef-formulas} to compute coefficients $X_{i,j}$ for
$i=0,1,2,3$, $j=0,1,2$;
\item use the $X_{i,j}$ to create quadratic polynomials
$$Q_i(x)=X_{i,2}\cdot x^2+X_{i,1}\cdot x+X_{i,0} \quad \text{for }
0\leq i \leq3,$$
and compute all 16 quadruples $(\alpha_0, \alpha_1, \alpha_2, \alpha_3)$
where each $\alpha_i$ is a root of $Q_i$;
\item for each of the above quadruples, let $z_3 = \sqrt{\alpha_3}$ and
for $i<3$ set $z_i = \pm \sqrt{\alpha_i}$ where the sign is chosen to agree
with the sign of $p^{\orig}_i \cdot p^{\orig}_3$;
\item among the 16 quadruples $(z_0,z_1,z_2,z_3)$, choose the one which
minimizes the error in the equations
$$a_i=b_j z_j^2 + b_k z_k^2 - 2 d_i z_j z_k
\qquad c_i=z_3^2 + b_i z_i^2 - 2 z_i z_3,$$
for $i=0,1,2$, with $j,k$ derived from $i$ as above;
\item set
$$z^{\orig}_i = \frac{||p^{\orig}_3||}{p^{\orig}_i \cdot p^{\orig}_3} z_i$$
for $i=0,1,2,3$.
\end{enumerate}

Over the next two sections, we will explain where this algorithm comes from,
and why it works to produce the claimed $z$-depths.

%
\section{Setting up the equations for the noiseless case}\label{s:first-equations}
%

In order to work toward our goal of an explicit formula for solving
the noiseless case of the perspective problem, we first have to set up
equations that govern the situation.
To say that the $n=4$ case of the perspective problem is overdetermined
means that among all possible sets of quadruples of points in $\BR^3$
together with quadruples of points on an image plane, only a
smaller-dimensional subset has the property that the 3D points can be
matched up with the image plane points after a suitable rigid
transformation. It turns out that this subset is an algebraic variety: i.e.,
it can be defined in terms of some polynomial equations. Finding one or
more of these equations would already be extremely helpful in the
perspective problem, as it would in principle allow for extremely efficient
seed rejection without attempting to solve for pose. Ultimately, we will
see that this constellation of ideas leads us to efficient
algorithms for solving for the pose as well. With this motivation,
in this section we set up our situation from the point of view of algebraic
geometry, seeing that the basic constraints arising in the perspective
problem can be expressed as polynomial equations
when we replace our initial tuples of points with tuples of distances and
dot products.
For this purpose, it is cleaner to work with lines through the origin (i.e.,
as points in the projective plane) in place of points on an image plane.

Our approach is motivated by the following basic fact, which is easily
verified but we recall here for convenience:

\begin{prp}\label{prp:tetrahedron} A (possibly degenerate) tetrahedron
in $\BR^3$ is determined uniquely up to rigid transformation by the
set of distances between the vertices.
\end{prp}

With this motivation in mind, the particular coordinates we will work with
have already been defined above,
but we reiterate here, again using cyclic indexing as before:

Define a map
\begin{equation}\label{eq:pts-dists}
(\BR^3)^4 \to \BR^6_{(c_0, c_1, c_2, a_0, a_1, a_2)}
\end{equation}
by setting
$$c_i = ||P_i - P_3||^2, \quad a_i = ||P_j - P_k||^2$$
for each $i$. Note that because the map is defined in terms of distances,
it is invariant under any rigid transformation applied simultaneously to
the quadruple of points.

Next, let 
$$\cU_0 \subseteq (\BR\BP^2)^4$$
be the open subset consisting of quadruples 
$(L_0,L_1,L_2,L_3)$ of lines in $\BR^3$ through
the origin such that none of
the $L_i$ is orthogonal to $L_3$. We also define a map
\begin{equation}\label{eq:lines-dists}
\cU_0 \to \BR^6_{(b_0, b_1, b_2, d_0, d_1, d_2)}
\end{equation}
by setting
$$b_i = \frac{(L_i \cdot L_i)(L_3 \cdot L_3)}{(L_i \cdot L_3)^2},\quad
d_i = \frac{(L_j \cdot L_k)(L_3 \cdot L_3)}{(L_j \cdot L_3)(L_k \cdot L_3)}$$
for each $i$.
To evaluate the above, we have to choose a nonzero representative on each
line for the individual dot products to make sense, but
note that the formulas are invariant under scaling each $L_i$, so the result
is independent of the choice of representatives. 

To further explain the definitions of the $b_i$ and $d_i$, note that the
formulas are invariant under simultaneous rotation. If we rotate all the $L_i$
together so that $L_3$ lies along the optical axis, and let $p_i$ be the 
point on the $z=1$ canvas obtained from each of the rotated $L_i$, we will
have $p_3=(0,0,1)$ and $p_i \cdot p_3 = 1$ for each $i$, so we obtain
simpler formulas: $b_i = p_i \cdot p_i$ and $d_i = p_j \cdot p_k$ for
each $i$.

The main result of the section is then that using the coordinates we have
introduced, if we introduce four additional variables we can obtain algebraic
constraints that must be satisfied by any tuple coming from the perspective
problem. This result also explains the meaning of the $z_i$ we introduced
in \S \ref{s:algorithm} above.

\begin{prp} \label{prp:equations}
Suppose that 
$(P^{\orig}_0,P^{\orig}_1,P^{\orig}_2,P^{\orig}_3,L^{\orig}_0,L^{\orig}_1,L^{\orig}_2,L^{\orig}_3)
\in \left(\BR^3\right)^4 \times \left(\BR\BP^2\right)^4$
has the properties that none of the $L^{\orig}_i$ is orthogonal to $L^{\orig}_3$, and there
exists a rigid transformation $T$ such that for each $i$, we have that
$T(P^{\orig}_i) \neq 0$ and $T(P^{\orig}_i)$ lies on $L^{\orig}_i$. Let
$(c_0,c_1,c_2,a_0,a_1,a_2,b_0,b_1,b_2,d_0,d_1,d_2)$ be the image of this tuple
under \eqref{eq:pts-dists} $\times$ \eqref{eq:lines-dists}.
Then there exist nonzero $z_0,z_1,z_2,z_3$ such that the following equations
are satisfied:
\begin{equation}\label{e:incidence}
a_i=b_j z_j^2 + b_k z_k^2 - 2 d_i z_j z_k
\qquad c_i=z_3^2 + b_i z_i^2 - 2 z_i z_3
\end{equation}
for $i=0,1,2$, and again using cyclic indexing for $j,k$.

More precisely:
\begin{enumerate} 
\item the desired $z_i$ can be obtained as the $z$-depths of the $T(P^{\orig}_i)$ 
relative to the image plane at distance $1$ from the origin, and with $L^{\orig}_3$ 
as its optical axis;
\item if $p_i$ denotes the image of $L^{\orig}_i$ on the above-mentioned image plane
for each $i$, then for a given choice of $z_i$ all the equations
\eqref{e:incidence} will be satisfied precisely if the condition
$$||P^{\orig}_i-P^{\orig}_j||^2 = ||z_i p_i - z_j p_j||^2$$
is satisfied for all $i,j$.
\end{enumerate}
\end{prp}

Note that in the last two statements, the $z_i$ are only determined up to
simultaneous sign change, since there are two possibilities for the image
plane described in the proposition. However, simultaneous sign change does not affect the equations
\eqref{e:incidence}.

\begin{proof} Since the map \eqref{eq:lines-dists} is invariant under 
simultaneous rotation, we may apply a rotation $R$ to the $L^{\orig}_i$ so that 
$L^{\orig}_3$ agrees with the optical axis without changing any of the $b_i, d_i$. 
Similarly, since the map \eqref{eq:pts-dists} is invariant under rigid
transformations, we may apply $T$ and then $R$ to the $P^{\orig}_i$ without changing
any of the $a_i, c_i$. That is, without changing any of the coordinates of
our proposed equations, we could have assumed that $P^{\orig}_i$ lies on $L^{\orig}_i$ for
each $i$, and that $L^{\orig}_3$ is the optical axis. Now, if we let 
$p_0,p_1,p_2,p_3=(0,0,1)$ be the representatives of the $L^{\orig}_i$ on the $z=1$
plane, there exist nonzero $z_i$ so that $P^{\orig}_i = z_i p_i$ for each $i$. These
are visibly the asserted $z$-depths, so it is enough to check that the
desired equations must all be satisfied.

But following through the definitions we have
\begin{align*} a_i = ||P^{\orig}_j - P^{\orig}_k||^2 
& = (z_j p_j - z_k p_k) \cdot (z_j p_j - z_k p_k) \\
& = z_j^2 b_j + z_k^2 b_k - 2 z_j z_k d_i,
\end{align*}
and noting that $p_3 \cdot p_i = 1$ for all $i$,
\begin{align*} c_i = ||P^{\orig}_i - P^{\orig}_3||^2 
& = (z_i p_i - z_3 p_3) \cdot (z_i p_i - z_3 p_3) \\
& = z_i^2 b_i + z_3^2 - 2 z_i z_3,
\end{align*}
as claimed.
\end{proof}

For completeness, we also record how to transform the $z_i$ above into
$z$-depths for the original canvas, justifying the final algorithm step
in \S \ref{s:algorithm}:

\begin{prp}\label{prp:z-transform} In the situation of \Cref{prp:equations},
if we let $p^{\orig}_i$ be the representatives of each $L^{\orig}_i$ on 
the usual canvas at $z=1$, and fix the choice of signs of the $z_i$ so that
$z_3>0$, then the $z$-depths of each $T(P^{\orig}_i)$ with respect 
to this canvas are given by 
$$\frac{||p^{\orig}_3||}{p^{\orig}_i \cdot p^{\orig}_3} z_i.$$
\end{prp}

Note that in contrast to the previous case, the $z$-depths of
\Cref{prp:z-transform} must all be positive, since in the original
perspective problem all the $P^{\orig}_i$ ought to have been in front of the
camera. 

\begin{proof} This is the same as saying that
$$T(P^{\orig}_i) = \frac{||p^{\orig}_3||}{p^{\orig}_i \cdot p^{\orig}_3} z_i p^{\orig}_i;$$
since we are just trying to compute a scalar, and since $p^{\orig}_i \cdot p^{\orig}_3$ is
nonzero by hypothesis, it is enough to check this after taking the dot 
product with $p^{\orig}_3$ and cancelling the denominator, which is to say it is enough
to show that 
$$T(P^{\orig}_i) \cdot p^{\orig}_3 = ||p^{\orig}_3|| z_i.$$

On the other hand, saying that $z_i$ is the $z$-depth of $T(P^{\orig}_i)$ relative
to the canvas with $L^{\orig}_3$ as its optical axis means that $z_i$ can be
obtained as the dot product with the normalization of any representative of
$L^{\orig}_3$; that is, we have 
$$z_i = T(P^{\orig}_i) \cdot \frac{p^{\orig}_3}{||p^{\orig}_3||},$$
which is what we wanted. We need only address the ambiguity of choice of
signs, but since
the only possibility is replacing $p^{\orig}_3$ by $-p^{\orig}_3$ as the representative for
$L^{\orig}_3$, it suffices to check that we have the right sign in the above
equation for $i=3$. Moreover, $T(P^{\orig}_3) \cdot p^{\orig}_3>0$ since both points lie in
front of the camera by hypothesis, so the fact that we fixed signs so that
$z_3>0$ means that the above equation has the correct sign, as desired.
\end{proof}

We conclude this section with a couple of remarks regarding mathematical
context for our work; these may be skipped by the reader who is only
interested in understanding our algorithm.

\begin{rmr}\label{rmr:cayley-menger} In order to better understand the images
of the maps \eqref{eq:pts-dists} and \eqref{eq:lines-dists}, it is helpful to
recall the notion of {\em Cayley-Menger determinants}. Given a quadruple
$(P_0,P_1,P_2,P_3)$ of points in $\BR^3$, the associated Cayley-Menger
determinant is given by
\[\left|\begin{smallmatrix}
  0&||P_0-P_1||^2&||P_0-P_2||^2&||P_0-P_3||^2&1\\
  ||P_0-P_1||^2&0&||P_1-P_2||^2&||P_1-P_3||^2&1\\
  ||P_0-P_2||^2&||P_1-P_2||^2&0&||P_2-P_3||^2&1\\
  ||P_0-P_3||^2&||P_1-P_3||^2&||P_2-P_3||^2&0&1\\
  1&1&1&1&0
\end{smallmatrix}\right|,\]
and -- up to multiplying by a positive scalar -- gives a formula for the
square of the volume of the tetrahedron spanned by the $P_i$.
It turns out that the image of \eqref{eq:pts-dists} is described precisely
as the set of $6$-tuples of prospective squared distances which are
nonnegative, whose square roots satisfy all the appropriate triangle
inequalities,\footnote{allowing for non-strict inequalities since we have not required
the $P_i$ to be distinct.} and which satisfy the condition that the
Cayley-Menger determinant is nonnegative. The first and last conditions
are explicitly semialgebraic in the $a_{\bullet}$ and $c_{\bullet}$,
and although each individual triangle inequality is not algebraic in
these coordinates, each triple of triangle inequalities coming from a fixed
subset of $\{0,1,2,3\}$ can be expressed in terms of a $2$-dimensional
Cayley-Menger determinant, so imposes a semialgebraic condition.

Turning to \eqref{eq:lines-dists}, because the $p^{\orig}_i$ 
are coplanar, the Cayley-Menger determinant with $p^{\orig}_i$ in place
of $P_i$ must vanish, which in our coordinates translates to the condition
that 
\begin{equation}\label{planarCM}
  \left|\begin{smallmatrix}
  0&b_0+b_1-2d_2&b_0+b_2-2d_1&b_0-1&1\\
  b_0+b_1-2d_2&0&b_1+b_2-2d_0&b_1-1&1\\
  b_0+b_2-2d_1&b_1+b_2-2d_0&0&b_2-1&1\\
  b_0-1&b_1-1&b_2-1&0&1\\
  1&1&1&1&0
  \end{smallmatrix}\right|=0.
\end{equation}

In principle, this should be relevant to us as it will impose an additional
polynomial equation on the set of interest to us, but in practice we have not
found it to be helpful in our elimination of variables.
\end{rmr}

\begin{rmr}\label{r:nonalgebraic} In the proof of \Cref{prp:equations}, we
can re-express our equations in
an invariant way in terms of the $P^{\orig}_i$ and $L^{\orig}_i$ as follows: since
$p_3=L^{\orig}_3/||L^{\orig}_3||$ and $p_i = L^{\orig}_i ||L^{\orig}_3||/(L^{\orig}_3\cdot L^{\orig}_i)$, we find that
$P^{\orig}_i = z_i p_i$ become
$p_3 = z_3 L^{\orig}_3 / ||L^{\orig}_3||$ and $P^{\orig}_i = z_i L^{\orig}_i ||L^{\orig}_3||/(L^{\orig}_3 \cdot L^{\orig}_i)$ for
$i>0$. We can clear denominators, but are left with
the square root implicit in $||L^{\orig}_3||$ appearing in each equation.
We see that although the equation
\[P^{\orig}_i = z_i p_i\]
looks innocuous and algebraic, it is in fact not algebraic ``upstairs'' in
$$\left(\BR^3\smallsetminus (0)\right)^4 \times \left(\BR\BP^2\right)^4
\times \BR^4.$$
Thus, one advantage of using the maps \eqref{eq:pts-dists} and
\eqref{eq:lines-dists} is that they convert non-algebraic equations of
interest into algebraic equations.
\end{rmr}

%
\section{Solving the equations for the noiseless case}\label{s:solving}
%

In the previous section, we obtained algebraic constraints coming from
the perspective problem after introducing the additional variables $z_i$.
We would like to obtain constraints in the original variables, which
requires eliminating the $z_i$s from our equations. Moreover, since the
$z_i$s are in fact $z$-depths, we also see that it is very useful to
be able to solve for them in terms of the other variables: this lets us
reduce perspective problems to absolute orientation problems. In theory,
elimination of variables is straightforward and can be accomplished by
plugging our polynomial equations into a computer algebra package such as
\verb|Singular|. However, in practice any system of equations which occurs
in a real-life problem will be too intractable to solve in this way. It
turns out that our problem lies on the boundary of tractability: if we just
throw in all our equations and ask it to eliminate all the $z_i$, it will
not succeed, but by experimenting with different variations on the process
we were able to get results. The key result that drives our algorithm is
the following.

\begin{thm}\label{thm:main-qi} There exist quadratic polynomials
$$Q_i(x)=X_{i,2}\cdot x^2+X_{i,1}\cdot x+X_{i,0} \quad \text{for }
0\leq i \leq3$$
with each
$X_{i,j}$ being an explicit polynomial in
$a_{\bullet}, b_{\bullet}, c_{\bullet}, d_{\bullet}$, having the following
property:
if a point of $\BR^6 \times \BR^6 \times \BR^4$ is in the zero set
of the equations \eqref{e:incidence}, then if we evaluate 
the polynomial $Q_{i}(z_i^2)$ on the coordinates of this point we will
likewise get $0$. 

The formulas for the $X_{i,j}$ are given in \Cref{app:coef-formulas}.
\end{thm}

Put differently, the value of $z_i^2$ obtained from our point must be a root
of the quadratic polynomial $Q_i(x)$, after evaluating the coefficients of
$Q_i(x)$ at the relevant values of 
$a_{\bullet}, b_{\bullet}, c_{\bullet}, d_{\bullet}$ 
in order to obtain a usual quadratic polynomial. 

\begin{proof} The polynomial $Q_0(z_0^2)$ is obtained via the \verb|Singular|
code in \Cref{app:singular} by starting from the ideal $I$ generated by
the six polynomials of \eqref{e:incidence}, and eliminating $z_3$, then
eliminating $z_2$, passing to a smaller ideal, and finally eliminating $z_1$.
The fact that $Q_0(z_0^2)$ is quadratic in $z_0^2$ is determined by
inspection. Because $Q_0(z_0^2)$ is in $I$, it has the property that it
vanishes at any point on which all the equations of \eqref{e:incidence}
vanish, as desired. $Q_3(z_3^2)$ is obtained via analogous steps, so we also
obtain the desired statements for it.

We next make a more general observation: if $\sigma$ is any permutation of
the set $\{0,1,2\}$, and we let $\sigma$ act on the variables 
$(a_{\bullet}, b_{\bullet}, c_{\bullet}, d_{\bullet}, z_{\bullet})$
simply by permutating the indices, so that only $z_3$ is fixed, then
although $\sigma$ doesn't fix the individual equations of \eqref{e:incidence},
it does permutate them, so it maps the ideal $I$ to itself. Then for $j=1,2$,
if we take the transposition $\tau_{0,j}$ which switches $0$ and $j$,
we see that $Q_0(z_0^2)$ will be sent to a polynomial which can be expressed
in terms of the $a_{\bullet}, b_{\bullet}, c_{\bullet}, d_{\bullet}$ and
$\tau_{0,j}(z_0^2)=z_j^2$, and which is moreover still quadratic in $z_j^2$.
We thus can define $Q_j(z_j^2)$ to be the image of $Q_0(z_0^2)$ under
$\tau_{0,j}$. Moreover, since $\tau_{0,j}$ sends $I$ to $I$, we know that
$Q_j(z_j^2)$ is still in $I$, so the desired vanishing statement holds
as well.
\end{proof}

\Cref{thm:main-qi} nearly completes the explanation of our algorithm.
The most significant remaining step is justifying that minimizing the
error in the equations \eqref{e:incidence} should be effective in
finding the ``correct'' tuple of $z_i$s out of the 16 possibilities
derived from the $Q_i$s. To this end, the following is helpful:

\begin{lma}\label{lma:linear} The ideal $I$ generated by the equations of
\eqref{e:incidence} also contains elements $\ell_i$ for $i=0,\dots,3$,
which are each polynomials in the 
$a_{\bullet}, b_{\bullet}, c_{\bullet}, d_{\bullet}$ and $z_i^2$, and
linear in $z_i^2$. 

Moreover, these polynomials are nonzero when evaluated
at general points coming from the perspective problem.
\end{lma}

\begin{proof}
The idea for obtaining linear expressions is simple enough: 
we apply the transposition $\tau_{1,2}$ to the coefficients of $Q_0$ in
the same manner as we derived $Q_1$ and $Q_2$ from $Q_0$. But because
$\tau_{1,2}$ keeps $0$ fixed, we will instead get a second quadratic
expression for $z_0^2$, and we can take a combination of these two to
cancel the leading terms and obtain a linear expression for $z_0^2$,
which will be our $\ell_0$.

Similarly, we can then obtain the remaining linear expressions analogously:
$\ell_1$ comes from applying $\tau_{0,2}$ to $Q_1$ and cancelling, $\ell_2$
from applying $\tau_{0,1}$ to $Q_2$ and cancelling, and $\ell_3$ from
applying $\tau_{1,2}$ to $Q_3$ and cancelling.

It remains to justify that our $\ell_i$ do not vanish identically on
points coming from perspective problem tuples, but this is addressed by
the following example.
\end{proof}

\begin{ex}\label{ex:example} Consider the example where we are given
$$P^{\orig}_0 = (0,0,0), \quad P^{\orig}_1 = (1,0,0), \quad
P^{\orig}_2 = (1,1,0), \quad P^{\orig}_3 = (0,0,3),$$
and canvas points
$$(2,1,1), \quad\scriptstyle{\left(\frac{17}{13},\frac{9}{13},1\right)}, \quad
{\scriptstyle\left(\frac{11}{15},\frac{4}{5},1\right)}, \quad
{\scriptstyle\left(\frac{1}{2},-\frac{11}{16},1\right)},$$
so that the lines $L^{\orig}_i$ generated by these are
$$L^{\orig}_0=t(2,1,1), \quad L^{\orig}_1=t(17,9,13), \quad
L^{\orig}_2=t(11,12,15), \quad
L^{\orig}_3=t(8,-11,16).$$
Then
$a_0=||P^{\orig}_1-P^{\orig}_2||^2=1$,
$a_1=||P^{\orig}_0-P^{\orig}_2||^2=2$,
$a_2=||P^{\orig}_0-P^{\orig}_1||^2=1$,
$c_0=||P^{\orig}_0-P^{\orig}_3||^2=9$,
$c_1=||P^{\orig}_1-P^{\orig}_3||^2=10$, and
$c_2=||P^{\orig}_2-P^{\orig}_3||^2=11$.
We also compute that $b_0=6$, $b_1=99/25$, $b_2=45/8$, $d_0=9/2$, $d_1=21/4$,
and $d_2=24/5$ (note that the above canvas points are not normalized to get
$p_3=(0,0,1)$, so we cannot apply the simplified dot product formulas to
get the $b_i$ and $d_i$; rather, we use the invariant forms in terms of
the lines $L^{\orig}_i$).

If we calculate our linear equations for the $z_i^2$,
we obtain the following equations:
\footnote{We did this calculation by using \texttt{Singular} to evaluate
the coefficients of each $Q_i$ and each $\tau_{j,k} Q_i$, and then to
calculate the gcd of $Q_i$ and $\tau_{j,k} Q_i$.}
$$z_0^2-1=0,\quad 9z_1^2-25=0, \quad 9z_2^2-16=0, \quad z_3^2-9=0.$$
Recall that these $z_i$ are $z$-depths for the normalized $p_i$, where
$p_3=(0,0,1)$.
Plugging into \eqref{e:incidence} we find that all the $z_i z_j$ are
positive, so we should choose all positive signs for the $z_i$
(i.e., $z_0=1$, $z_1=5/3$, $z_2=4/3$, $z_3=3$). Using these $z$-depths
with the normalized $p_i$ does indeed reproduce the same
square distances for the $z_i p_i$ as we have for the $P^{\orig}_i$, as desired.
%
\end{ex}

\begin{cor}\label{cor:algorithm-works} Suppose we are given general points
$$(P^{\orig}_0,p^{\orig}_0),(P^{\orig}_1,p^{\orig}_1),(P^{\orig}_2,p^{\orig}_2),(P^{\orig}_3,p^{\orig}_3)$$
where each $P^{\orig}_i \in \BR^3$ and $p^{\orig}_i$ is a canvas point at $z=1$, and suppose
further that
none of the $p^{\orig}_i$ is orthogonal to $p^{\orig}_3$, and there
exists a rigid transformation $T$ such that for each $i$, we have that
$T(P^{\orig}_i) \neq 0$ and $T(P^{\orig}_i)$ lies on the line connecting the origin to $p^{\orig}_i$.
	
Then (in the absence of
noise and rounding error) the algorithm described in \S \ref{s:algorithm}
will produce the $z_i$ such that $T(P^{\orig}_i) = z_i p^{\orig}_i$ for each $i$.
\end{cor}

\begin{proof} According to \Cref{prp:equations}, if $\bar{z}_i$ are the
$z$-depths of the $T(P^{\orig}_i)$ relative to the transformed canvas with $p^{\orig}_3$ on
its optical axis, then we have that the equations \eqref{e:incidence} vanish
when the $\bar{z}_i$ are plugged in along with the values of
$a_{\bullet}, b_{\bullet}, c_{\bullet},d_{\bullet}$ obtained from the $P^{\orig}_i$
and $p^{\orig}_i$. It follows from \Cref{thm:main-qi} that the $\bar{z}_i^2$
occur as one of the tuples of $\alpha_i$ in our algorithm, and we claim that
the $\bar{z}_i$ are then obtained as one of the 16 tuples of $z_i$ in
the next step (i.e., that the algorithm gives the correct signs for
taking the square roots). To justify this, note that the signs in the
algorithm visibly guarantee that the final $z^{\orig}_i$ are all positive. On
the other hand, \Cref{prp:z-transform} says that the transformation from
the algorithm will convert the $\bar{z}_i$ to the desired $z$-depths,
which are positive by hypothesis since the $T(P^{\orig}_i)$ must all lie in front
of the camera. This both confirms that the algorithm is indeed finding
the $\bar{z}_i$ as one of its 16 tuples, and simultaneously justifies that
the final step of the algorithm gives the desired $z$-depths. It thus
suffices to verify that the $\bar{z}_i$ will be the tuple chosen in the
``error minimization'' step. Since we have assumed no noise or rounding
error, the tuple $\bar{z}_i$ will satisfy the equations exactly, and we
just need to know that none of the other tuples also satisfy the equations.
This is where the hypothesis that the $P^{\orig}_i$ and $p^{\orig}_i$ are general
comes in, since we can finally use \Cref{lma:linear} to conclude that the
$\bar{z}_i^2$ are uniquely determined by the vanishing of 
\eqref{e:incidence}, so the $\bar{z}_i$ will be the unique tuple which
minimizes the error.
\end{proof}


We conclude this section by presenting some polynomial constraints on
quadruples of matched 3D points and canvas points which must be satisfied
in order for the associated $4$-point perspective problem to be solvable.
As mentioned early, this is potentially of substantial practical interest,
since it offers another approach to rapid culling of quadruples containing
mismatched points. At present, we have found it more practical to simply
apply our main algorithm in order to simultaneously recognize mismatches
and solve for $z$-depths, but due to the combination of mathematical and
potentially algorithmic interest, we include a fairly complete discussion.
The derivation of these constraints is straightforward given what we have
already accomplished:

\begin{dfn} For each $i$, let $g_{i,i}$ be the rational function in
$a_{\bullet},b_{\bullet},c_{\bullet},d_{\bullet}$ obtained by setting
$\ell_i=0$ and solving for $z_i^2$. For each pair of distinct $i,j$,
let $g_{i,j}$ be the rational function in the same variables obtained from the
appropriate equation in \eqref{e:incidence} as the expression for $z_i z_j$
after isolating $z_i z_j$ on one side and plugging in $g_{k,k}$ for each of
the $z_k^2$.
\end{dfn}

\begin{prp}\label{prp:var-eqs} The desired polynomial constraints on matched
tuples are obtained by setting
$$g_{i,i} g_{j,j}=(g_{i,j})^2$$
for each $i<j$, and
$$g_{i,j} g_{j,k}=g_{i,k} g_{j,j}$$
for all distinct $i,j,k$, and clearing denominators as necessary.

Moreover, our constraints are generically complete in the following precise
sense: if our matched tuple has the property that when we evaluate the
$a_{\bullet}, b_{\bullet}, c_{\bullet}, d_{\bullet}$ for this tuple,
none of the $d_i$ are zero, and none of the denominators of the $g_{i,i}$ 
are zero, and if the above polynomial constraints are all satisfied by our
matched tuple, then the tuple yields a solvable $4$-point perspective problem.
\end{prp}

It is also worth noting that we see from \Cref{ex:example} that a
generic matched tuple (even one coming from a solvable perspective problem)
will satisfy the nonvanishing conditions on the $d_i$ and denominators of
the $g_{i,i}$, so we really can say that our constraints are complete for
general matched tuples arising from perspective problems.

\begin{proof}
If we start with a matched tuple that corresponds to a solvable perspective
problem, and let $z_i$ be the transformed $z$-depths as in
\Cref{prp:equations}, we conclude from \Cref{prp:equations} and \Cref{lma:linear}
that we must have $z_i^2=g_{i,i}$ for each $i$, and the conditions
$z_i z_j=g_{i,j}$ for $i,j$ distinct are likewise obtained from
the equations \eqref{e:incidence}.
We then obtain the desired equations from the identities
$z_i^2 z_j^2 = (z_i z_j)^2$ and
$(z_i z_j)(z_j z_k) = (z_i z_k) z_j^2$ respectively.

In order to verify the assertion that our equations are generically complete,
it is conceptually helpful to introduce variables $z_{i,j}$ for all $i,j$
intended to represent putative values of $z_i z_j$.
From this point of view, we see that under our assumption that the 
$d_{\bullet}$ and denominators of the $g_{i,i}$ are all nonzero,
imposing our original equations on the original set of variables gives the
same results as introducing the $z_{i,j}$, and then imposing that
$z_{i,j}=g_{i,j}$ for all $i,j$, together with
\begin{equation} \label{e:1st-minors}
z_{i,i} z_{j,j} = \left(z_{i,j}\right)^2
\end{equation}
for all $i,j$ distinct, and
\begin{equation} \label{e:2nd-minors}
z_{i,j} z_{j,k} = z_{i,k} z_{j,j}
\end{equation}
for all $i,j,k$ distinct. Note that setting $z_{i,j}=g_{i,j}$ is
possible precisely because of our nonvanishing assumption.

Now, suppose we have a tuple of $(P^{\orig}_i, p^{\orig}_i)$ such that the corresponding
values of $a_{\bullet},b_{\bullet},c_{\bullet},d_{\bullet}$ satisfy the
equations from the statement of the proposition. Then setting
$z_{i,j}=g_{i,j}$ for all $i,j$, we claim we can produce values for
$z_i$ such that $z_{i,j}=z_i z_j$ for all $i \leq j$.
Indeed, if we choose $z_i$ to be any square root of $z_{i,i}$ for each $i$,
then equation \eqref{e:1st-minors}
guarantees that we have $z_{i,j} = z_i z_j$ up to signs. If all the $z_i$
are zero, there is nothing more to check, but if some $z_j \neq 0$, we can
fix a sign choice for $z_j$, and then determine the signs of $z_i$ for
$i \neq j$ by setting $z_i = z_{i,j}/z_j$ (again, this still satisfies
$z_i^2 =z_{i,i}$ due to \eqref{e:1st-minors}). Then we conclude that
$z_{i,k} = z_i z_k$ for all $i,k$ not equal to $j$ by applying
Equation \eqref{e:2nd-minors}, proving the claim.

Next, by construction the equations \eqref{e:incidence}
are automatically satisfied, so it follows from the last part of
\Cref{prp:equations} that after renormalizing to account for the different
image canvas we can produce depths $z^{\orig}_i$ such that the distances
between the $P^{\orig}_i$ agree with the distances between the $z^{\orig}_i p^{\orig}_i$.
It follows from \Cref{prp:tetrahedron} that the perspective problem is
solvable.
\end{proof}

We next make some remarks on the algebraic geometry context of the equations
we have been studying.

\begin{rmr} 
Denote by $W\subset\BR^6\times\BR^6\times\BR^4$ 
the zero locus of the equations in \eqref{e:incidence}, and let
$V$ to be the image of $W$ when we project to 
$\BR^6\times\BR^6$ (i.e., forgetting the $z_{\bullet}$).

Note that $V$ is larger than the set of points coming from perspective
problems, since we have ignored the Cayley-Menger
determinant condition discussed in \Cref{rmr:cayley-menger}. However,
this is ultimately irrelevant to our work.

Observe that geometrically, the
equations defining $W$ can be viewed as realizing $W$ as the graph of
a map $\BR^{10} \to \BR^6$, since they are giving formulas for the
$a_{\bullet}, c_{\bullet}$ in terms of the other variables. Thus,
$W$ is isomorphic to $\BR^{10}$, and in particular is irreducible of
dimension $10$. Since $V$ is defined as the image of $W$, it follows
that $V$ is also irreducible, of dimension at most $10$.

Moreover, \Cref{ex:example} shows that at least some fibers of the map 
$W \to V$ are finite; from Chevalley's theorem it follows that the fibers
are finite on a dense open subset of $W$, so $\dim V = \dim W = 10$.
\end{rmr}

\begin{rmr}
In theory, any given one of the equations we describe in \Cref{prp:var-eqs}
above could turn out to be trivial, if we happen to produce the 
same formulas on both sides. However, since we know that we have produced
enough equations to determine $V$, and since $V$ has codimension $2$, we know
without even computing the equations that
we are at least producing some nontrivial constraints on the variables.

Along similar lines, even if all the $Q_i$ and the linear polynomials 
$g_{i,i}$ are 
known not to be identically zero, we would also theoretically have to worry
that they could vanish identically on the points of interest, coming from
perspective problems. This would happen if for instance they were multiples
of the Cayley-Menger determinant for the canvas points.

Of course, we have seen
from real-world experience that this is not the case, but it is also ruled
out definitively by \Cref{ex:example}, where we start with actual $4$-tuples
of points, with the second tuples being coplanar, and verify that we get
nontrivial linear equations for the $z_i^2$.
\end{rmr}

We conclude with a contextual remark on motivation and possible alternate
mathematical approaches.

\begin{rmr}\label{rmr:algebraic_comments}
Our coordinates choice is motivated by invariant theory: This is more evident on the 3D side: the
$a_i, c_i$ are the invarants under (the 6-dimensional group $E(3)$ of) rigid
transformations of the 12-dimensional space representing four points in $\BR^3$. On the 2D side the
$b_i, d_i$ are the invariants under the 1D group of rotations about the
optical axis of the 6 dimensional space
of planar configurations of 3 points.

We originally tried to act with the full group of rotations on the 2D side. This would give
much more symmetric equations in place of \eqref{e:incidence}, but on the other hand, it means moving
to a spherical canvas representation. This proved computationally infeasible, probably
due to the non-algebraicity explained in \Cref{r:nonalgebraic}.

Another direction is to try to ``slice'' the moduli space instead of taking invariants; we already
partially did this for the 2D points by taking a representative for the configuration for which
one point is on the optical axis. However, one can do much more: e.g. on the 3D side force one of
the points to be at the origin, another to be on the optical axis, and another on some plane containing
this axis. This did not prove computationally fruitful (specifically, one gets rather ugly equations).
\end{rmr}

%
\section{Utilization for the perspective $n$-points problem}\label{s:perspective}
%

We begin by summarizing the entire process of solving a perspective
$n$-points problem using our algorithm.

Suppose we are given $n$ prospective matched pairs
between distinct $3D$ points and distinct canvas points
(imbedded at $z=1$), with $n \geq 4$. Then:

\begin{enumerate}
\item Choose a suitable $N$ based on $n$ and the amount of computational
time available, and choose $N$ random subsets
$S_1,\dots,S_N \subseteq \{1,\dots,n\}$, each of size $4$.

\item For each $m \leq N$, set up the perspective $4$-points problem
obtained by letting $P^{\orig}_0,\dots,P^{\orig}_3$ be the set of 3D points
coming from $i \in S_m$,
and $p^{\orig}_0,\dots,p^{\orig}_3$ the corresponding canvas points.
Then using our $4$-point algorithm as described in \S \ref{s:algorithm},
compute expected $z$-depths $z^{\orig}_0,\dots,z^{\orig}_3$ to make
the point quadruple $z^{\orig}_j p^{\orig}_j$ into a rigid transformation of
the $P^{\orig}_j$.

%
%
%
%
%
%
\item Iteratively unite subsets $S_m$ of points, which have estimated error below a certain
  threshold,
and which overlap with a previously united subset on three of the four indices,
and which
have strong agreement of the renormalized $z^{\orig}_i$ for these three indices.

\item Apply Horn's algorithm \cite{Ho} to solve for the absolute orientation
for each of the most promising united subsets of point matches, as measured
by size and cumulative error.

\item For each pose obtained from Horn's algorithm, apply Fletcher's variant
of the Levenberg-Marquardt algorithm \cite{F} to use gradient descent to
further minimize reprojection error.

\item Return to the full original set of matched points, and choose the
pose candidate which has the best combination of agreement with large
numbers of points matches, and small reprojection error.
\end{enumerate}

We conclude with some contextual remarks on practical computational
considerations.

\begin{rmr}\label{rmr:coef-compute}
The topic of actually evaluating complex multivariate polynomials such as
the coefficients of the $Q_i$ is a matter of current research. All approaches
are by and large inspired by Horner's classical
algorithm, where one typically builds a directed acyclic evaluation graph
whose non-terminal nodes are of the form
{\small \begin{verbatim}
Current_node_value = some_variable * child0_value + child1_value,
\end{verbatim}}
\noindent
and where the terminal nodes are either variables or numbers.
This form of evaluation plays nicely with the standard hardware machine operation
``fused multiply and add''.
The evaluation graphs we built for the coefficients
$X_{0,0},X_{0,1},X_{0,2},X_{3,0},X_{3,1},X_{3,2}$ have
$113, 150, 122, 120, 132, 83$ nonterminal nodes respectively. No doubt less naive approaches
than what we tried would yield better results.
\end{rmr}

\begin{rmr}\label{rmr:why_quad}
We prefer to find the $z_i^2$ by solving the $Q_is$ rather than using the
$\ell_i$, since the latter would either mean evaluating the much more
complex symbolic expressions for the coefficients of $\ell_i$, or
doubling the number of $X_{i,j}$ we have to evaluate in order to compute
two quadratic polynomials and cancel the leading terms. In contrast, solving
the quadratics amounts just to taking a square root and keeping track of the
combinatorics of the $2^4$ solutions.

One may argue that following the same reasoning we could have dealt with
a quartic in $z_0^2$. In fact, it is rather easy to find this quartic just
by eliminating $z_1, z_2$ from the 3 equations in \eqref{e:incidence}
not involving $z_3$: this is the quartic
describing the $2\times 4$ solutions to the perspective 3 points problem.
The reason we prefer the quadratic is threefold: solving a quartic equation
involves many roots, which is much more complicated to keep track of than
a single sign; the algebraic computation is far less stable; and finally, by using an overdetermined
system we get an accuracy measure.
\end{rmr}

\begin{rmr}\label{rmr:solution}
With real life data, due to the presence of measurement, as well as numerical
inaccuracies, our set of equations cannot be exactly solved. As always, we
seek the solution which minimizes some error model. The standard error model
for the perspective $n$ points model is the {\em reprojection error}: the
sum of the squares of the distances between the points on the canvas, and the
projection of the 3D points on the canvas. Minimizing the reprojection error
cannot be done exactly, but rather is solved with some optimization algorithm.

One can view the sum
of the difference of the evaluations of \verb|f1|...\verb|f6| after substituting
in the
estimations of $z_i$ as the ``error'' in our solution. Although this is not the
reprojection error or the reconstruction error from SQPnP, it is something
which lies, in a sense, in the same domain: namely it is the
difference between the squared distances between the points in 3D, and the
squared distances between the selected points along the rays.
While our error is not identical to the reprojection error, it is usually close
enough so that for real data we get a ``seed solution'' which is closer to 
the optimized solution than we would obtain from EPnP type solutions. For
a precise comparison, see \S \ref{s:experiments} below.

To get a more balanced error evaluation, it is better to solve all the $z_i^2$
from their respective $Q_i$, although if solution stability were not an issue,
given even just one, we can derive linear equations
for the rest in terms of rational functions in the one we solved for.

Finally and crucially, note that if the error for a configuration is large,
the configuration
may be discarded without solving the absolute orientation problem at all.
\end{rmr}

\begin{rmr}
Observe that all the operations involved in computing $z_i^2$ -- including taking a
square root -- are vectorizable on any modern architecture. This fact give
a large performance boost -- more than an order of magnitude --
compared to gradient descent algorithms such as EPnP. 
\end{rmr}

%
\section{Experiments and Evaluation}\label{s:experiments}
%
\subsection{Experimental Setup}\ \\
We evaluated our algorithm -- with three different thresholds for the estimated error -- against EPnP and SQPnP across the following parameters:
\begin{description}
\item [Computational efficiency] Average execution time per configuration (the variance
is extremely small for all algorithms involved).
\item [Rotation and translation accuracy] Angular and Eculidean distances between the computed solution and ground truth rotations and translations.
\item [Number of successes = true positive rate] Percentage of trials the algorithm returned a 
  solution when one is expected.
\item [Early true negative rate] Percentage of trials the algorithm failed fast when there isn't
  supposed to be a solution.
\end{description}
We conducted two experiments. In the main one we tested the first three paramaeters.
To this end, we conducted multiple measurements, where in each measurement we generated
a synthetic scenario with the following characteristics:
\begin{description}
\item [3D Points] We performed three different sub-experiments: ``general setting'', using
  four random points on the unit sphere; ``planar setting'', using four points on the
  unit circle, and ``three collinear points'' where two points are at $(\pm1, 0, 0)$, one is normally
  sampled on the line connecting the first two, and one is randomly sampled on the unit sphere.
\item [Random rotation] Chosen randomly on SO(3), then applied to our 3D points.
\item [Random translation] A random point on the unit sphere, used to translate the four 3D points.
\item [Noise model] Add ``noise factor'' times a random point on the unit sphere to each of the four points.
\item [Projected points] The original, unaltered 3D points were translated 2.5 units in the direction of the optical axis to ensure they were in front of the camera, and then projected to the image plane.
\end{description}
For the second experiment, which checks for ``fast rejection'' capabilities, we used
the same data generation scheme above, except that one of the 3D points is subsequently replaced with
a different random point chosen from the same distribution.

Our algorithm was implemented using a batched SIMD approach for efficient computation of
compatible configurations (compiled with -O3 for our tests; our algorithm performs 45\% faster when
compiled to AVX2, but then there are linkage issues with openCV).
Recall from the introduction that our algorithm outputs depth values for the four planar points,
which are then combined with Horn's method to recover the full 6-DOF pose.
For comparison, we used the faster of OpenCV's 4.6.0 and 4.13.0-dev implementations of EPnP and
SQPnP.
\subsection{Results and Analysis}
\subsubsection{Computational efficiency:}
Measured in $\mu$-seconds per configuration on a 5 GHz 13th Gen Intel(R) Core(TM) i7-1360P
\begin{description}
\item [EPnP] taking $n=4$ in OpenCV 4.6.0 implementation: 25.771
\item [SQPnP] taking $n=4$ in OpenCV 4.13.0-dev implementation: 36.312
\item [Our algorithm] 0.477
\item [Our algorithm] when compiled with -mAVX2 flag: 0.258
\item [Horn] (naive implementation using OpenCV 4.6.0) :2.533
\end{description}
\subsubsection{Rotations and translation accuracy, and successes}
For each noise level and configuration type, we conducted 10,000 measurements, summed up in the table below:
{\fontsize{4}{8} \selectfont
  \begin{center}
    \begin{tabular}{r|rrrrr|rrrrr|rrrrr}
      \multicolumn{1}{c|}{noise}&\multicolumn{5}{c|}{rotation error in degrees; avg(std)}&\multicolumn{5}{c|}{translation error in milli-units; avg(std)}&\multicolumn{5}{c}{success count}\\
      \multicolumn{1}{c|}{milli}&EPnP&SQPnP&\multicolumn{3}{c|}{ours for the following thresholds}&EPnP&SQPnP&\multicolumn{3}{c|}{our for the following thresholds}&EPnP&SQPnP&\multicolumn{3}{c}{ours}\\
      units&n=4&n=4&0.05&0.1&1&n=4&n=4&0.05&0.1&1&n=4&n=4&0.05&0.1&1\\
      \hline
      \multicolumn{16}{c}{general configuration}\\
\hline
 0&      11.4(24.3)&  1.8(14.6)&  0.5(2.8)&   0.9(4.3)&    6.7(24.4)&       277(822)&  24(218)&   8(47)&   15(71)&   104(374)&     10000&   9999&    7884&    8200&    9387\\
 1&      11.2(24.1)&  1.8(14.2)&  1.0(4.1)&   1.7(6.8)&    8.4(26.8)&       257(712)&  23(193)&  17(66)&   28(103)&  122(382)&     10000&   9998&    7421&    7955&    9372\\
 2&      11.7(24.8)&  1.8(14.0)&  1.4(5.5)&   2.1(7.3)&    8.6(25.8)&       266(525)&  25(207)&  25(98)&   37(129)&  131(377)&     10000&   9998&    7001&    7762&    9342\\
 3&      11.4(23.7)&  2.3(15.9)&  1.7(5.4)&   2.5(8.0)&    9.8(27.4)&       259(518)&  31(216)&  29(87)&   42(126)&  148(404)&     10000&   9997&    6566&    7453&    9275\\
 \hline
 4&      11.6(24.5)&  2.2(15.2)&  1.8(5.1)&   2.7(7.6)&   10.0(26.8)&       269(667)&  32(214)&  33(90)&   48(124)&  158(410)&     10000&   9999&    6267&    7327&    9295\\
 5&      11.7(24.1)&  2.4(15.4)&  2.1(5.5)&   3.0(7.8)&   10.4(27.5)&       281(1072)& 35(212)&  36(90)&   51(124)&  165(420)&     10000&  10000&    5975&    7137&    9237\\
 6&      11.5(23.7)&  2.2(14.2)&  2.3(6.1)&   3.5(9.2)&   11.6(29.3)&       265(587)&  35(210)&  40(106)&  59(149)&  181(439)&     10000&   9999&    5639&    7016&    9249\\
 8&      12.1(24.9)&  2.5(15.1)&  2.8(6.9)&   4.2(10.3)&  12.5(29.2)&       274(573)&  40(214)&  49(121)&  72(170)&  197(442)&     10000&  10000&    5116&    6673&    9204\\
 \hline
10&      12.4(24.6)&  3.2(17.0)&  3.0(7.0)&   4.4(10.2)&  12.5(28.3)&       275(557)&  51(245)&  54(118)&  78(176)&  198(428)&     10000&   9997&    4719&    6413&    9125\\
12&      12.6(24.7)&  3.3(16.5)&  3.4(8.1)&   5.0(10.9)&  13.8(29.7)&       286(540)&  53(231)&  59(117)&  85(173)&  222(456)&     10000&  10000&    4352&    6166&    9185\\
15&      13.1(25.5)&  3.4(16.1)&  3.5(6.7)&   5.2(10.5)&  14.8(30.6)&       289(529)&  57(222)&  63(115)&  89(172)&  235(460)&     10000&   9997&    3838&    5732&    9135\\
20&      13.9(26.6)&  3.9(15.9)&  4.5(8.2)&   6.3(12.4)&  16.8(32.3)&       309(545)&  69(231)&  80(138)& 110(204)&  272(500)&     10000&   9999&    3234&    5231&    9041\\
\hline
25&      14.1(25.9)&  4.4(16.2)&  5.4(10.8)&  7.5(14.9)&  18.2(32.7)&       309(675)&  79(226)&  94(148)& 128(218)&  301(514)&     10000&   9998&    2752&    4714&    8914\\
30&      14.4(26.0)&  5.0(16.8)&  6.0(10.7)&  8.1(14.4)&  19.7(33.8)&       319(644)&  90(243)& 110(183)& 144(241)&  316(506)&     10000&   9998&    2387&    4328&    8896\\
\hline
\multicolumn{16}{c}{planar configuration}\\
\hline
0&  24.0(19.1)&   7.2(16.8)&   8.0(12.2)&  12.2(15.5)&  16.3(20.3)&  401(400)&    132(356)&    121(205)&    187(266)&    256(363)&     10000&   10000&    7154&    8939&    9916\\
5&      26.3(18.5)&  10.9(19.2)&  12.0(13.6)&  16.1(16.4)&  19.8(21.0)&       431(396)&    192(396)&    176(224)&    238(276)&    301(372)&     10000&   10000&    6598&    8831&    9716\\
10&      27.0(18.3)&  13.1(19.6)&  13.3(13.8)&  16.8(15.9)&  20.8(20.7)&       451(403)&    236(417)&    199(233)&    254(276)&    318(371)&     10000&   10000&    6389&    8607&    9678\\
20&      28.0(18.2)&  17.5(20.5)&  15.2(13.6)&  18.2(15.2)&  23.0(21.1)&       464(401)&    309(446)&    225(232)&    271(265)&    353(390)&     10000&   10000&    5816&    8201&    9579\\
\hline
\multicolumn{16}{c}{3 collinear points + one extra point}\\
\hline
0&      43.5(52.5)&   9.3(32.1)&   2.2(7.5)&   3.0(9.0)&  13.1(29.0)&       552(649)&    107(373)&     32(108)&     44(130)&    211(489)&     10000&   10000&    7317&    7721&    9331\\
5&      49.3(54.3)&  12.3(34.6)&   5.8(12.3)&   7.0(14.2)&  18.4(32.4)&       613(651)&    148(408)&     81(163)&    100(190)&    285(517)&     10000&   10000&    6151&    7062&    9247\\
10&      50.5(54.8)&  15.7(37.9)&   7.2(13.3)&   9.6(17.4)&  21.8(34.2)&       637(769)&    187(442)&    100(172)&    133(221)&    331(530)&     10000&   10000&    5315&    6522&    9119\\
20&      50.1(53.5)&  20.8(41.3)&   9.7(15.1)&  12.1(18.5)&  25.7(36.0)&       628(644)&    252(476)&    136(192)&    170(235)&    385(545)&     10000&   10000&    4481&    5906&    8958\\
    \end{tabular}
    \end{center}
}

We note several conclusions deduced from the data in the table:
\begin{itemize}
\item The high standard deviation of the error when compared to the average error -- in all the
  algorithms -- is due to the fact that they are all somewhat unstable. Note however,
  that the quotient of the
  standard deviation of the error by the average error is lower in our algorithm than in SQPnP.
\item Looking at general configurations, with threshold .05 our algorithm is -- on average -- as accurate as SQPnP,
  which is the state of the art for accuracy. With threshold 1 it is as accurate as EPnP.
\item Our algorithm suffers much less than EPnP and SQPnP from degenerate configurations, which are
  ubiquitous in real life data.
\item Even with a huge error rate of 30 milli-units (which is 3\% noise) we still have
25\% of extremely accurate solutions (these are the cases where the errors all ``align perfectly'' so
to speak to represent a ``good'' configuration).
\end{itemize}
Regarding the realism of the different noise levels: while 3\% noise rate is theoretically
interesting, even 1\%=10 milli-units noise translates to several centimeters in a typical room, or 5
pixels in a standard modern non-HD 640$\times$480 pixels video frame. In practice if ones errors are higher than this, one has bigger
problems than the speed and accuracy of the chosen algorithm.
\subsubsection{Early true negative}
In our ``fast rejection'' experiment, EPnP solved the configurations with average rotational error 75.2(48.1) and average translation error 1292(974) with 10000 successes,
SQPnP solved the configurations with average rotational error 72.8(48.7) and average translation error 1252(912) with 9996 successes, whereas our algorithm rejected (i.e. ``failed'') 99\%
of the configurations with threshold .05, and 96\% of them with threshold 0.1. Thus our success criterion present a huge advantage in cases where the underlying 2D/3D data is accurate, but
point matching is poor.

%
\section{Conclusion}\label{s:conclusion}
%

By drastically speeding up the process of solving for pose
and providing a clear and effective criterion for seed rejection, our
algorithm has the potential to be transformative in real-world
perspective problems. In any situation our algorithm compares favorably to
the current
state of the art, running two orders of magnitude faster than EPnP or SQPnP 
to reduce a problem to solving for absolute orientation, and one order of
magnitude faster even when including the time for Horn's algorithm to fully
solve for pose. Its accuracy is roughly comparable as well: depending on
where one sets the desired error threshold, it can range from similar to
EPnP to similar to SQPnP, at the cost of rejecting a higher percentages
of quadruples for stricter thresholds. Moreover, our algorithm shows a
high degree of stability, and is resilient to degenerate configurations
such as coplanar points, or triples of collinear points, both of which can
frequently occur with real-world data.

Beyond its general versatility, our algorithm especially shines in 
situations where accurate point matching is challenging, and one needs to
sift through potentially thousands of matched point quadruples in order to
find ones of high enough quality to use for accurate pose reconstruction.
Such situations are quite common, and our algorithm's ability to reject
badly matched quadruples fully two orders of magnitude faster than EPnP or 
SQPnP places it in a class by itself, and allows for processing far more
quadruples than was previously possible when run time is a limiting factor.


\appendix
\section{Singular calculations}\label{app:singular}

In order to obtain the quadratic polynomials $Q_i$ referenced in 
\Cref{thm:main-qi}, we begin by simply coding the equations
in \eqref{e:incidence} into \verb|Singular|, and letting $I$ be the
ideal generated by all six polynomials.
{\small
\begin{verbatim}
ring r=0,(c0,c1,c2,a0,a1,a2,b0,b1,b2,d0,d1,d2,z0,z1,z2,z3),(dp(12),dp(4));
poly f1 = b1*z1^2+b2*z2^2-2*d0*z1*z2-a0;
poly f2 = b2*z2^2+b0*z0^2-2*d1*z2*z0-a1;
poly f3 = b0*z0^2+b1*z1^2-2*d2*z0*z1-a2;
poly f4 = z3^2+b0*z0^2-2*z0*z3-c0;
poly f5 = z3^2+b1*z1^2-2*z1*z3-c1;
poly f6 = z3^2+b2*z2^2-2*z2*z3-c2;
ideal I=f1,f2,f3,f4,f5,f6;
\end{verbatim}
}
The chosen polynomial ordering is natural due to our desire to eliminate
the last four variables from the ideal.

To compute $Q_0$, the first two steps are straightforward: eliminate $z_3$
and then $z_2$ from the ideal $I$.
{\small
\begin{verbatim}
ideal I0_3=eliminate(I, z3);  // very fast
ideal I0_2=eliminate(I0_3, z2);  // seconds
size(I0_2);  // returns 13
\end{verbatim}
}
However, trying to directly eliminate $z_1$ from the above ideal hit a
computational brick wall;
instead, after some experimentation, we defined a certain subideal, and
were able to make the final elimination starting from it:
{\small
\begin{verbatim}
ideal J0_2=I0_2[1],I0_2[3],I0_2[5];  // Taking a sub-ideal
ideal J0_1=eliminate(J0_2, z1);  // this one takes a few minutes
size(J0_1);  // returns 13, J0_1[1] is degree 2 in z0^2
\end{verbatim}
}
In defining \verb|J0_2|, our ultimate choices are in fact
rather natural: the three polynomials in the definition of
\verb|J0_2| are simply the shortest three generators of \verb|I0_2|, and the
reason for using three of them is that this is
the minimal number which
would be expected to support eliminating $z_1$ while still having a
nontrivial equation in $z_0$.

In any case, the final line of the above \verb|Singular| code shows that
we have produced an element of $I$ of the desired form, so we thus define
$Q_0$:

\begin{dfn} Let $Q_0(x)$ be the polynomial obtained by replacing \verb|z0|$^2$
with $x$ in the above \verb|J0_1[1]|, considered as a quadratic
polynomial in $x$ with coefficients being polynomials in the
$a_\bullet, b_\bullet, c_\bullet, d_\bullet$.
\end{dfn}

To obtain $Q_3$ we run the following code in the
same spirit as our derivation of $Q_0$:
{\small
\begin{verbatim}
ideal I3_2=eliminate(I, z2);
ideal I3_1=eliminate(I3_2, z1);
ideal J3_1=I3_1[1], I3_1[3], I3_1[4];  // Taking a sub-ideal again
ideal J3_0=eliminate(J3_1, z0);
\end{verbatim}
}
We then get that 
\verb|J3_0[1]| is of degree $2$ in \verb|z3^2|, so we can define $Q_3$:
\begin{dfn} Let $Q_3(x)$ be the polynomial obtained by replacing \verb|z3|$^2$
with $x$ in the above \verb|J3_0[1]|, considered as a quadratic
polynomial in $x$ with coefficients being polynomials in the
$a_\bullet, b_\bullet, c_\bullet, d_\bullet$.
\end{dfn}

%
%
%
%

\section{The coefficients of the $Q_i$}\label{app:coef-formulas}

Below are the formulas for the coefficients of $Q_0$ and $Q_3$.
The output of our \verb|Singular| code was post processed to remove the
\verb|*| symbol:
{\tiny
\begin{verbatim}
X00 = c0^2a0b1d0d1-2c0c1a0b1d0d1+c1^2a0b1d0d1+c0^2a1b1d0d1-2c0c1a1b1d0d1+c1^2a1b1d0d1-c0^2a2b1d0d1+2c0c1a2b1d0d1-c1^2a2b1d0d1
+2c0a0a2b1d0d1-2c1a0a2b1d0d1+2c0a1a2b1d0d1-2c1a1a2b1d0d1-2c0a2^2b1d0d1+2c1a2^2b1d0d1+a0a2^2b1d0d1+a1a2^2b1d0d1-a2^3b1d0d1
-c0^2a0b1b2d2+2c0c1a0b1b2d2-c1^2a0b1b2d2-c0a0^2b1b2d2+c1a0^2b1b2d2+c0^2a1b1b2d2-2c0c1a1b1b2d2+c1^2a1b1b2d2+2c0a0a1b1b2d2
-2c1a0a1b1b2d2-c0a1^2b1b2d2+c1a1^2b1b2d2+c0^2a2b1b2d2-2c0c1a2b1b2d2+c1^2a2b1b2d2-a0^2a2b1b2d2+2a0a1a2b1b2d2-a1^2a2b1b2d2
+c0a2^2b1b2d2-c1a2^2b1b2d2+a0a2^2b1b2d2-a1a2^2b1b2d2-2c0^2a1d0^2d2+4c0c1a1d0^2d2-2c1^2a1d0^2d2+2a1a2^2d0^2d2-c0a0^2b1b2
-c1a0^2b1b2+2c0a0a1b1b2+2c1a0a1b1b2-c0a1^2b1b2-c1a1^2b1b2+2c0a0a2b1b2+2c1a0a2b1b2+a0^2a2b1b2-2c0a1a2b1b2-2c1a1a2b1b2
-2a0a1a2b1b2+a1^2a2b1b2-c0a2^2b1b2-c1a2^2b1b2-2a0a2^2b1b2+2a1a2^2b1b2+a2^3b1b2+4c0a1a2d0^2+4c1a1a2d0^2-4a1a2^2d0^2
-4c0a0a2d0d1-4c0a1a2d0d1+4c0a2^2d0d1+2c0a0^2b2d2-4c0a0a1b2d2+2c0a1^2b2d2-2c0a2^2b2d2

X01 = -4c0a0b0b1d0d1+4c1a0b0b1d0d1-4c0a1b0b1d0d1+4c1a1b0b1d0d1+4c0a2b0b1d0d1-4c1a2b0b1d0d1-4a0a2b0b1d0d1-4a1a2b0b1d0d1+4a2^2b0b1d0d1
-2c0^2b0b1b2d2+4c0c1b0b1b2d2-2c1^2b0b1b2d2+2a0^2b0b1b2d2-4a0a1b0b1b2d2+2a1^2b0b1b2d2-4c0a2b0b1b2d2+4c1a2b0b1b2d2-4a0a2b0b1b2d2
+4a1a2b0b1b2d2+2c0^2b0d0^2d2-4c0c1b0d0^2d2+2c1^2b0d0^2d2+4c0a1b0d0^2d2-4c1a1b0d0^2d2-4a1a2b0d0^2d2-2a2^2b0d0^2d2+2c0^2b1d1^2d2
-4c0c1b1d1^2d2+2c1^2b1d1^2d2+4c0a0b1d1^2d2-4c1a0b1d1^2d2+4a0a2b1d1^2d2-2a2^2b1d1^2d2-4c0^2d0d1d2^2+8c0c1d0d1d2^2-4c1^2d0d1d2^2
+4a0a2d0d1d2^2+4a1a2d0d1d2^2+2c0^2b2d2^3-4c0c1b2d2^3+2c1^2b2d2^3-2a0^2b2d2^3+4a0a1b2d2^3-2a1^2b2d2^3-4c0a0b0b1b2-4c1a0b0b1b2
+4c0a1b0b1b2+4c1a1b0b1b2+4c0a2b0b1b2+4c1a2b0b1b2+4a0a2b0b1b2-4a1a2b0b1b2-4a2^2b0b1b2-4c0a1b0d0^2-4c1a1b0d0^2-4c0a2b0d0^2
-4c1a2b0d0^2+4a1a2b0d0^2+4a2^2b0d0^2+4c0a0b0d0d1+4c0a1b0d0d1-4c0a2b0d0d1+4a0a2b0d0d1+4a1a2b0d0d1-4a2^2b0d0d1-4c1a0b1d0d1
-4c1a1b1d0d1+4c1a2b1d0d1+4a0a2b1d0d1+4a1a2b1d0d1-4a2^2b1d0d1+4c0a0b1d1^2+4c1a0b1d1^2-4c0a2b1d1^2-4c1a2b1d1^2-4a0a2b1d1^2
+4a2^2b1d1^2+4c0a0b0b2d2-2a0^2b0b2d2-4c0a1b0b2d2+4a0a1b0b2d2-2a1^2b0b2d2+4c0a2b0b2d2+2a2^2b0b2d2+4c1a0b1b2d2-2a0^2b1b2d2
-4c1a1b1b2d2+4a0a1b1b2d2-2a1^2b1b2d2-4c1a2b1b2d2+2a2^2b1b2d2+8c1a1d0^2d2+8c0a2d0d1d2+8c1a2d0d1d2-8a0a2d0d1d2-8a1a2d0d1d2
-8c0a0d1^2d2+4a0^2b2d2^2-8a0a1b2d2^2+4a1^2b2d2^2-4c0a2b2d2^2-4c1a2b2d2^2

X02 = 4a0b0^2b1d0d1+4a1b0^2b1d0d1-4a2b0^2b1d0d1+4c0b0^2b1b2d2-4c1b0^2b1b2d2+4a0b0^2b1b2d2-4a1b0^2b1b2d2-4c0b0^2d0^2d2
+4c1b0^2d0^2d2+4a2b0^2d0^2d2-4c0b0b1d1^2d2+4c1b0b1d1^2d2-8a0b0b1d1^2d2+4a2b0b1d1^2d2+8c0b0d0d1d2^2-8c1b0d0d1d2^2-4a0b0d0d1d2^2
-4a1b0d0d1d2^2-4a2b0d0d1d2^2-4c0b0b2d2^3+4c1b0b2d2^3-4a0b0b2d2^3+4a1b0b2d2^3+8a0d1^2d2^3-4c0b0^2b1b2-4c1b0^2b1b2+4a2b0^2b1b2
+4c0b0^2d0^2+4c1b0^2d0^2-4a2b0^2d0^2-4a0b0^2d0d1-4a1b0^2d0d1+4a2b0^2d0d1-4a0b0b1d0d1-4a1b0b1d0d1+4a2b0b1d0d1+4c0b0b1d1^2
+4c1b0b1d1^2-4a2b0b1d1^2-4a0b0^2b2d2+4a1b0^2b2d2-4a2b0^2b2d2+8c1b0b1b2d2-4a0b0b1b2d2+4a1b0b1b2d2-4a2b0b1b2d2-8c1b0d0^2d2
-8c0b0d0d1d2-8c1b0d0d1d2+8a0b0d0d1d2+8a1b0d0d1d2+8a0b0d1^2d2-8c1b1d1^2d2+8a0b1d1^2d2+4c0b0b2d2^2+4c1b0b2d2^2+8a0b0b2d2^2
-8a1b0b2d2^2+4a2b0b2d2^2+16c1d0d1d2^2-16a0d1^2d2^2-8c1b2d2^3

X30 = c0^2c1b0b1b2+c0c1^2b0b1b2-c0^2c2b0b1b2+c1^2c2b0b1b2-c0c2^2b0b1b2-c1c2^2b0b1b2+c0^2a1b0b1b2-c1^2a1b0b1b2+2c0c2a1b0b1b2
+2c1c2a1b0b1b2-c0a1^2b0b1b2-c1a1^2b0b1b2-c0^2a2b0b1b2-2c0c1a2b0b1b2-2c1c2a2b0b1b2+c2^2a2b0b1b2+2c1a1a2b0b1b2-2c2a1a2b0b1b2
+a1^2a2b0b1b2+c0a2^2b0b1b2+c2a2^2b0b1b2-a1a2^2b0b1b2-c0^3b0b1d1-2c0^2c1b0b1d1-c0c1^2b0b1d1+c0^2c2b0b1d1+2c0c1c2b0b1d1+c1^2c2b0b1d1
+c0^2a1b0b1d1+2c0c1a1b0b1d1+c1^2a1b0b1d1+2c0^2a2b0b1d1+2c0c1a2b0b1d1-2c0c2a2b0b1d1-2c1c2a2b0b1d1-2c0a1a2b0b1d1-2c1a1a2b0b1d1
-c0a2^2b0b1d1+c2a2^2b0b1d1+a1a2^2b0b1d1+2c0^2c2b1d1^2-2c1^2c2b1d1^2+4c1c2a2b1d1^2-2c2a2^2b1d1^2+c0^3b0b2d2-c0^2c1b0b2d2
+2c0^2c2b0b2d2-2c0c1c2b0b2d2+c0c2^2b0b2d2-c1c2^2b0b2d2-2c0^2a1b0b2d2+2c0c1a1b0b2d2-2c0c2a1b0b2d2+2c1c2a1b0b2d2+c0a1^2b0b2d2
-c1a1^2b0b2d2-c0^2a2b0b2d2-2c0c2a2b0b2d2-c2^2a2b0b2d2+2c0a1a2b0b2d2+2c2a1a2b0b2d2-a1^2a2b0b2d2-4c0^2c2d1^2d2+4c0c1c2d1^2d2
+4c0c2a2d1^2d2-2c0^2c1b2d2^2+2c1c2^2b2d2^2-4c1c2a1b2d2^2+2c1a1^2b2d2^2+4c0^2c1d1d2^2-4c0c1c2d1d2^2-4c0c1a1d1d2^2

X31 = -4c0c1b0b1b2-2c1^2b0b1b2+4c0c2b0b1b2+2c2^2b0b1b2-4c0a1b0b1b2-4c2a1b0b1b2+2a1^2b0b1b2+4c0a2b0b1b2+4c1a2b0b1b2-2a2^2b0b1b2
+4c0^2b0b1d1+4c0c1b0b1d1-4c0c2b0b1d1-4c1c2b0b1d1-4c0a1b0b1d1-4c1a1b0b1d1-4c0a2b0b1d1+4c2a2b0b1d1+4a1a2b0b1d1-2c0^2b1d1^2
+2c1^2b1d1^2-4c0c2b1d1^2+4c1c2b1d1^2-4c1a2b1d1^2-4c2a2b1d1^2+2a2^2b1d1^2-4c0^2b0b2d2+4c0c1b0b2d2-4c0c2b0b2d2+4c1c2b0b2d2
+4c0a1b0b2d2-4c1a1b0b2d2+4c0a2b0b2d2+4c2a2b0b2d2-4a1a2b0b2d2+4c0^2d1^2d2-4c0c1d1^2d2+4c0c2d1^2d2-4c1c2d1^2d2-4c0a2d1^2d2
-4c2a2d1^2d2+2c0^2b2d2^2+4c0c1b2d2^2-4c1c2b2d2^2-2c2^2b2d2^2+4c1a1b2d2^2+4c2a1b2d2^2-2a1^2b2d2^2-4c0^2d1d2^2-4c0c1d1d2^2
+4c0c2d1d2^2+4c1c2d1d2^2+4c0a1d1d2^2+4c1a1d1d2^2-2c0^2b0b1+2c1^2b0b1+4c0a1b0b1+4c1a1b0b1-4c1a2b0b1-4a1a2b0b1+2a2^2b0b1
+2c0^2b0b2-2c2^2b0b2+4c2a1b0b2-2a1^2b0b2-4c0a2b0b2-4c2a2b0b2+4a1a2b0b2+2c1^2b1b2-2c2^2b1b2+4c2a1b1b2-2a1^2b1b2-4c1a2b1b2+2a2^2b1b2
-4c0^2b0d1+4c0c2b0d1+4c0a1b0d1+4c0a2b0d1-4c2a2b0d1-4a1a2b0d1-4c1^2b1d1+4c0c2b1d1+4c0a1b1d1+8c1a2b1d1-4a2^2b1d1+8c2a2d1^2+4c0^2b0d2
-4c0c1b0d2-4c0a1b0d2+4c1a1b0d2-4c0a2b0d2+4a1a2b0d2-4c0c1b2d2+4c2^2b2d2-8c2a1b2d2+4a1^2b2d2-4c0a2b2d2+8c0c1d1d2-8c0c2d1d2
-8c0a1d1d2+8c0a2d1d2-8c1a1d2^2

X32 = 4c1b0b1b2-4c2b0b1b2+4a1b0b1b2-4a2b0b1b2-4c0b0b1d1+4c2b0b1d1+4a1b0b1d1+4c0b1d1^2-4c1b1d1^2+4a2b1d1^2+4c0b0b2d2-4c1b0b2d2
-4a2b0b2d2-4c0d1^2d2+4c1d1^2d2+4a2d1^2d2-4c0b2d2^2+4c2b2d2^2-4a1b2d2^2+4c0d1d2^2-4c2d1d2^2-4a1d1d2^2+4c0b0b1-4c1b0b1-8a1b0b1
+4a2b0b1-4c0b0b2+4c2b0b2-4a1b0b2+8a2b0b2-4c1b1b2+4c2b1b2-4a1b1b2+4a2b1b2+4c0b0d1-4c2b0d1-4a1b0d1-4c0b1d1+8c1b1d1-4c2b1d1-4a1b1d1
-8a2b1d1-8a2d1^2-4c0b0d2+4c1b0d2+4a2b0d2+4c0b2d2+4c1b2d2-8c2b2d2+8a1b2d2+4a2b2d2-8c1d1d2+8c2d1d2+8a1d1d2-8a2d1d2+8a1d2^2
+8a1b0-8a2b0+8a1b1-8a2b2+16a2d1-16a1d2
\end{verbatim}
}

In order to obtain the formulas for $X_{1,j}$ and $X_{2,j}$, as described
in the proof of \Cref{thm:main-qi} one simply starts with the formulas for the
$X_{0,j}$, and applies a transposition to every variable index
occurring in the formula. For $Q_1$, all of the $0$s and $1$s are swapped,
while for $Q_2$ all of the $0$s and $2$s are swapped.

\end{document}